\documentclass[11 pt,a4paper]{amsart} 
\usepackage{amsfonts,amssymb,amscd,amsmath,enumerate,verbatim,calc} 
\usepackage{float}
\usepackage{amsthm}
\usepackage{xcolor}
\usepackage{titlesec}
\titleformat{\section}[block]{\large\bfseries\filcenter}{}{1em}{}
\titleformat{\subsection}[hang]{\bfseries}{}{1em}{}
\newtheorem{theorem}{Theorem}[section]
\newtheorem{definition}{Definition}[section]
\newtheorem{lemma}[theorem]{Lemma}
\newtheorem{corollary}[theorem]{Corollary}
\newtheorem{proposition}[theorem]{Proposition}
\newtheorem{remark}[theorem]{Remark}

\usepackage{mathtools}

\newtheorem{example}[theorem]{Example}

\numberwithin{equation}{section}

\usepackage{amsfonts}
\newcommand{\field}[1]{\mathbb{#1}}

                   \newcommand{\Z}{\field{Z}}

\renewcommand{\ker}{\textnormal{ker}}

\begin{document}

\title{Excision and idealization of a multiplicative Lie algebra}
	
\author{Neeraj Kumar Maurya$^{1}$, Amit Kumar$^{2}$ AND Sumit Kumar Upadhyay$^{3}$\vspace{.4cm}\\
		{$^{1,2,3}$Department of Applied Sciences,\\ Indian Institute of Information Technology Allahabad\\Prayagraj, U. P., India} }
	
	\thanks{$^1$neerajkr2699@gmail.com, $^2$amitiiit007@gmail.com,  $^3$upadhyaysumit365@gmail.com}
	\thanks{Corresponding author:  Sumit Kumar Upadhyay}

\begin{abstract}
In this article, we introduce the concepts of excision and idealization for a multiplicative Lie algebra (also for a Lie algebra), which provides two new multiplicative Lie algebras (or Lie algebras) from a given multiplicative Lie algebra (or Lie algebra) and an ideal, under certain conditions. These concepts may assist in classifying all multiplicative Lie algebras (or Lie algebras) of a specified order (or dimension).
\end{abstract}

\maketitle
\section{Introduction}
The main aim of this paper is to introduce the concept of excision multiplicative Lie algebra and the idealization of a multiplicative Lie algebra (a similar concept for a Lie algebra), motivated by the idea of an excision ring and the idealization of a ring. For a given ring without identity, in 1932, Dorroh \cite{DO} constructed a ring with identity in which that given ring is embedded. Anderson generalized this construction \cite[Theorem 2.1]{AND}  as follows: Let $R$ be a ring with identity and $A$ be a ring that is a unitary $R-R-$ bimodule. Then $R\oplus A$ is a ring with the following operations:
\begin{align*}
    (r_1, a_1)+(r_2, a_2) &= (r_1+r_2, a_1+a_2)\\
    (r_1, a_1)(r_2, a_2) &= (r_1r_2, r_1a_2+a_1r_2+a_1a_2).
\end{align*}
In particular, if $R$ is a commutative ring with identity and $I$ is an ideal of $R$, then $R\oplus I$ is called an excision ring. In \cite{NA}, there is the concept of idealization of a ring, which is given as follows:
Let $R$ be a ring with identity and $I$ be an ideal of $R$. Then $R\rtimes I$ is a ring with the following operations:
\begin{align*}
    (r_1, a_1)+(r_2, a_2) &= (r_1+r_2, a_1+a_2)\\
    (r_1, a_1)(r_2, a_2) &= (r_1r_2, r_1a_2+a_1r_2).
\end{align*}
By the motivation of the above constructions, the amalgamation of a ring with another ring along an ideal with respect to a ring homomorphism came into the picture (for details, see \cite{AN}).

The concept of multiplicative Lie algebra was introduced by Ellis \cite{GJ}  in 1993.  Since the multiplicative Lie algebra is a generalization of group and Lie algebra, the mathematician has started to explore the known concepts in groups and Lie algebras in the case of multiplicative Lie algebras. In 1996, Point and Wantiez \cite{FP} introduced the concept of nilpotency for multiplicative Lie algebras. In 2021, Pandey et al. \cite{MRS} discussed a different notion of solvable and nilpotent multiplicative Lie algebra. So many algebraic properties like homology theory, non-abelian tensor product, and Schur multiplier of a multiplicative Lie algebra have been studied in \cite{AGNM, GNM, GNMA, GM, RLS}. There is a question of how many multiplicative Lie algebra structures exist up to isomorphism on a given finite group. This question has been discussed for some cases in \cite{AD, MS, Walls}. The concept of excision and idealization may also help to answer this question. The idea of excision and idealization for Lie algebras may also be helpful to classify Lie algebras of a given dimension.  Now, we recall some basics of multiplicative Lie algebras. 

\begin{definition}\cite{GJ}
	A multiplicative Lie algebra is a triple $ (G,\cdot,\star), $ where $ (G,\cdot) $ is a group together with a binary operation $ \star $ on $G$ satisfying the following identities: 
	\begin{enumerate}
		\item $ x\star x=1 $  
		\item $ x\star(yz)=(x\star y){^y(x\star z)} $ 
		\item $ (xy)\star z= {^x(y\star z)} (x\star z) $ 
		\item $ ((x\star y)\star {^yz})((y\star z)\star{^zx})((z\star x)\star{^xy})=1 $ 
		\item $ ^z(x\star y)=(^zx\star {^zy})$ 
	\end{enumerate}
for all $x,y,z\in G$, where $^xy$ denotes $xyx^{-1}$. We say $ \star $ is a multiplicative Lie algebra structure on the group $G$.
\end{definition}
\begin{definition}\label{D1}
Let $(G,\cdot,\star)$ be a multiplicative Lie algebra. Then
\begin{enumerate}
\item A subgroup $H$ of $G$ is said to be a subalgebra of G if $x\star y \in H$ for all $x,y \in H$.

\item A subalgebra $H$ of $G$ is said to be an ideal of $G$ if it is a normal subgroup of $G$ and $x\star y \in H$ for all $x \in G$ and $y \in H$. 

\item The group center	$ Z(G) = \{x \in G \mid [x,y]=1 $ for all $y\in G\}$  is an ideal  of $G.$

\item The ideal	$ LZ(G) = \{x \in G \mid  x \star  y = 1 $ for all $y\in G\}$  is called the Lie center of $G.$ 

\item The ideal	$ MZ(G) = \{x \in G \mid  x \star  y = [x,y] $ for all $y\in G\}$  is called the multiplicative Lie center of $G.$ 

\item Let $(G',\circ , \star')$ be another multiplicative Lie algebra. A group homomorphism $\psi: G \to G'$ is called a multiplicative Lie algebra homomorphism if $\psi(x\star y) =\psi(x) \star ' \psi(y)$ for all $x, y \in G$.

\end{enumerate}
\end{definition}
\begin{definition}\cite{MRS}
    Let $(G_1,\cdot_1, \star_1)$ and  $(G_2,\cdot_2, \star_2)$ be multiplicative Lie algebras. Then, $(G_1\times G_2, \cdot, \star )$ is called the direct product of $G_1$ and $G_2$, where $(g_1,g_2)\cdot (g_3,g_4) = (g_1\cdot_1 g_3, g_2\cdot_2 g_4)$ and 
$(g_1,g_2)\star (g_3,g_4) = (g_1\star_1 g_3, g_2\star_2 g_4)$ for all $g_1, g_3\in G_1$ and $g_2, g_4\in G_2.$ 
\end{definition}
\section{Excision and Idealization}
In this section, we introduce the concept of an excision multiplicative Lie algebra and the idealization of a multiplicative Lie algebra.
\begin{lemma}\label{L1}
    Let $G$ be a group with a multiplicative Lie algebra structure $\star$, and let $I$ be an ideal of $G$ such that  $I\subseteq Z(G)$. Then $$^h(g\star b)={^hg\star b}=g\star b$$ for all $g,h \in G$ and $b\in I$. Also, $LZ(G)\cap I$ = $MZ(G)\cap I$.
\end{lemma}
\begin{proof}
    Since $(g\star b)\in I$ and  $I\subseteq Z(G)$, we have $(g\star b)={^h(g\star b)}={^hg}\star {^hb}= {^hg}\star b$. Now, let $x\in LZ(G)\cap I $. Then $x\star y = 1$ and $[x,y] = 1$ (since $I\subseteq Z(G)$) for all $y\in G$. This implies $x\star y = [x,y]$ for all $y\in G$. Hence $x\in MZ(G)\cap I $. Next, if $x\in MZ(G)\cap I $, it is easy to see that $x\in LZ(G)\cap I $. Therefore, we are done.
\end{proof}

\begin{theorem}\label{T1}
    Let $G$ be a group with a multiplicative Lie algebra structure $\star$, and let $I$ be an ideal of $G$ such that $I\subseteq Z(G)$. Then
\begin{enumerate}
        \item the set $G\oplus I = \{(g,a) : g\in G, a\in I\}$
        forms a multiplicative Lie algebra with respect to the following two binary operations
    \begin{align*}
    (g, a)\cdot (h,b) &:= (gh,ab) \\
    (g, a)\star' (h,b) &:= \big(g\star h, (g\star b)(a\star h)(a\star b)\big)
    \end{align*}
 for all $g,h\in G$ and $a,b\in I$. We call $(G\oplus I, \cdot, \star' )$ an excision multiplicative Lie algebra.
 \item the set $G\rtimes I = \{(g, a): g\in G, a\in I\}$ forms a multiplicative Lie algebra with respect to the following two binary operations
    \begin{align*}
    (g, a)\cdot (h,b) &:= (gh,ab) \\
    (g, a)\star'' (h,b) &:= \big(g\star h, (g\star b)(a\star h)\big)
    \end{align*}
 for all $g,h\in G$ and $a,b\in I$. We call $(G\rtimes I, \cdot, \star'' )$ as an idealization of $G$ by $I$ and denote it by $G\rtimes I$.
    \end{enumerate}
    
\end{theorem}

\begin{proof}
\begin{enumerate}[\itshape(1)]
 \item It is clear that $(G\oplus I, \cdot)$ forms a group. We need to verify that $\star'$ is a multiplicative Lie algebra structure.
 
\begin{itemize}
\item   $(g,a)\star' (g,a) = \big(g\star g, (g\star a)(a\star g)(a\star a)\big) = (1,1). $
\vspace{.25cm}
\item $(g,a)\star' ((h,b)(h',b')) = (g,a)\star'(hh',bb') 
    =\big(g\star hh', (g\star bb')(a\star hh')(a\star bb')\big)\\
    = \big((g\star h)^h(g\star h'), (g\star b)^b(g\star b')(a\star h)^h(a\star h')(a\star b)^b(a\star b')\big)  \\
    = \big(g\star h,(g\star b)(a\star h)(a\star b)\big) \big({^h(g\star h')},(g\star b')(a\star h')(a\star b')\big) ~~( \text{ By Lemma} ~\ref{L1})\\
     = \big(g\star h,(g\star b)(a\star h)(a\star b)\big) \big({^h(g\star h')}, {^b((g\star b')(a\star h')(a\star b'))}\big) ~~( \text{ By Lemma} ~\ref{L1})\\
    = ((g,a)\star' (h,b)) ^{(h,b)}((g,a)\star' (h',b'))$ .
    
\vspace{.25cm}    
Similarly, we can prove
\item   $((g,a)(g',a'))\star' (h,b) = {^{(g,a)}}((g',a')\star' (h,b)) ((g,a)\star' (h,b))$.
\vspace{.25cm}
\item $((g,a) \star' (h,b)) \star' {^{(h,b)}}(k,c) = ((g,a) \star' (h,b)) \star' (^hk,c) \\
= ((g\star h),(g \star b)(a\star h)(a\star b)) \star' (^hk,c)\\
= ((g\star h)\star {^hk},((g \star h)\star c)(((g \star b)(a\star h)(a\star b))\star {^hk})(((g \star b)(a\star h)(a\star b))\star c))\\
=(((g \star h)\star {^hk}), ((g \star h)\star c)((g \star b)\star {^hk})((a \star h)\star {^hk})((a \star b)\star {^hk})((g \star b)\star c)((a \star h)\star c)((a \star b)\star c))  ~~(\text{since}~ I \subseteq Z(G))\\
=(((g \star h)\star {^hk}), ((g \star h)\star c){^h}((g \star b)\star {k}){^h}((a \star h)\star {k}){^h}((a \star b)\star {k})((g \star b)\star c)((a \star h)\star c)((a \star b)\star c))  ~~( \text{ By Lemma} ~\ref{L1})\\
=(((g \star h)\star {^hk}), ((g \star h)\star c)((g \star b)\star {k})((a \star h)\star {k})((a \star b)\star {k})((g \star b)\star c)((a \star h)\star c)((a \star b)\star c))  ~~( \text{ By Lemma} ~\ref{L1})$.

Similarly, we have 

$((h,b) \star' (k,c)) \star' {^{(k,c)}}(g,a)= (((h \star k)\star {^kg}), ((h \star k)\star a)((h \star c)\star {g})((b \star k)\star {g})((b \star c)\star {g})((h \star c)\star a)((b \star k)\star a)((b \star c)\star a))$  \\ and \\
$((k,c) \star' (g,a)) \star' {^{(g,a)}}(h,b)= (((k \star g)\star {^gh}), ((k \star g)\star b)((k \star a)\star {h})((c \star g)\star {h})((c \star a)\star {h})((k \star a)\star b)((c \star g)\star b)((c \star a)\star b))$. 

Finally, we get

$((g,a) \star' (h,b)) \star' {^{(h,b)}}(k,c))((h,b) \star' (k,c)) \star' {^{(k,c)}}(g,a))((k,c) \star' (g,a)) \star' {^{(g,a)}}(h,b))=(1,1)$. 

\item $^{(k,c)}((g,a) \star' (h,b))= \big({^k(g\star h)}, (g\star b)(a\star h)(a\star b)\big)= {^{(k,c)}(g,a) \star' {^{(k,c)}(h,b)}} ~~$ (By Lemma ~\ref{L1}).
\end{itemize}
Hence, we proved $(G\oplus I, \cdot, \star' )$ a multiplicative Lie algebra.

\item Similarly, we can prove that $(G\rtimes I, \cdot, \star'' )$ is a multiplicative Lie algebra.
\end{enumerate}
\end{proof}

The following lemma shows that $G \oplus I$ and $G\rtimes I$  are not always isomorphic as multiplicative Lie algebras.
\begin{lemma}\label{l1}
    $G\oplus I$ is isomorphic (as a multiplicative Lie algebra) to $G\rtimes I$ if and only if $I\star I = 1$. Also, $J = \{(1, a) \mid a \in I\}$ is an ideal of $G\rtimes I$ such that $J \star'' J = 1$, that is, $J$ is a nilpotent ideal of $G\rtimes I$.
\end{lemma}
\begin{proof}
If $I\star I = 1$, then by Theorem \ref{T1},  $G\oplus I$ is isomorphic (as multiplicative Lie algebra) to $G\rtimes I$. Now, we prove the converse part. 

Consider the set $J = \{(1, a) \mid a \in I\}$. Then $(J, \cdot, \star')$ is an ideal of $G\oplus I$ and $(J, \cdot, \star'')$ is an ideal of $G\rtimes I$. It is clear that $(J, \cdot, \star'')$ is a trivial multiplicative Lie algebra. On the other hand, if $I\star I \ne 1$, then $(J, \cdot, \star')$ is an abelian group with a nontrivial multiplicative Lie algebra structure. So, if $G\oplus I$ is isomorphic (as multiplicative Lie algebra) to $G\rtimes I$, then $(J, \cdot, \star')$ should be isomorphic to $(J, \cdot, \star'')$,  which is a contradiction. Hence, $I\star I = 1.$
\end{proof}
\begin{proposition}
   The maps $p: G\oplus I \to G$ and $i: G\to G\oplus I$ defined by $p(g, a) = g$ and $i(g) = (g, 1)$ are  multiplicative Lie algebra homomorphisms with $p\circ i = Id_G$.
\end{proposition}
\begin{remark}\label{1}
\begin{enumerate}
\item $G\oplus Z(G)$ need not be isomorphic (as multiplicative Lie algebra) to $G\rtimes Z(G)$.
\item If the underlying group of a multiplicative Lie algebra $G$ is a nilpotent group of class $2$, then $[G, G] \star [G, G] = 1$. Hence, $G\oplus [G, G]$ is isomorphic (as multiplicative Lie algebra) to $G\rtimes [G, G]$.

\item If the underlying group of a multiplicative Lie algebra $G$ is abelian, then $G\oplus I$ and $G\rtimes I$ are multiplicative Lie algebras for any ideal $I$ of $G$.

 \item $G\oplus \mathcal{Z}(G)$, $G\rtimes \mathcal{Z}(G)$ and $G\times \mathcal{Z}(G)$ are all isomorphic, where $\mathcal{Z}(G) = Z(G) \cap LZ(G)$.
\end{enumerate}
 
\end{remark}
\subsection{Excision and idealization of a Lie algebra.} 
Let $L$ be a Lie algebra over a field $\mathbb{F}$ with a Lie bracket $[,]$, and let $M$ be an ideal of $L$.
\begin{enumerate}
    \item By Theorem \ref{T1} and Remark \ref{1} (3), $L\oplus M$ and $L\rtimes M$ are  also Lie algebras with the Lie brackets $[,]'$ and $[,]''$ defined as follows:
    \begin{align*}
        [(g,a), (h,b)]' &:= \big([g, h], [g, b][a, h][a, b]\big) \\
        [(g, a), (h,b)]'' &:= \big([g, h], [g, b][a, h]\big),
    \end{align*}
    respectively. In both cases, the underlying vector space structure is the direct product of vector spaces.   

\item Let $L$ be a Lie algebra of dimension $n$, and let $M$ be an ideal of $L$ of dimension $k$.  Then by Lemma \ref{l1}, $L\oplus M$ and $L\rtimes M$ are isomorphic for $k= 1.$ Also, $L\oplus M$ and $L\rtimes M$ are not isomorphic for $k\geq 2$, if induced Lie bracket on $M$ is non-trivial and $[L, M] \ne 0$.
\end{enumerate}


Next, we illustrate with an example that  $G \oplus I$ and $G\rtimes I$  are not necessarily isomorphic to the direct product $G \times I$ as multiplicative Lie algebra.
\begin{example}\label{E1}
Consider the Klein's four group $V_4 = \langle a,b:a^2 = b^2 = 1, ab = ba \rangle$ with the  multiplicative Lie algebra structure $ a\star b = a $ \cite[Theorem 2.5]{MS}. Let $I$ be the ideal  $V_4\star V_4 = \{1, a\}$ of $V_4$. Then by Lemma \ref{l1}, $V_4 \oplus I$ and $V_4 \rtimes I$ are isomorphic. Since  $(V_4 \times I)\star (V_4 \times I)\cong \Z_2$ and $(V_4 \oplus I)\star' (V_4 \oplus I)\cong V_4$, $V_4 \oplus I$ and $V_4 \times I$ are not isomorphic. 

\begin{remark}
We know that $V_4\cong \Z_2 \times \Z_2 $ is a $2$-dimensional Lie algebra over the field $\Z_2$ with respect to the bracket operation $[a,b]=a$ as given in Example \ref{E1}. Thus, by Example \ref{E1}, we have two distinct Lie algebras of dimension $3$, namely $V_4 \oplus I$ and $V_4 \times I$ over $\Z_2$.  
\end{remark}

\end{example}
\subsection{Structure of Ideals}
We know that subgroups of the direct product of two groups need not be the direct products of their subgroups. For example,  the subgroup $\Delta = \{(g, g) \mid g\in G\}$ of group $G\times G$ is not a direct product of subgroups of $G$. In this subsection, we will explore the structure of ideals in the context of excision and idealization for a multiplicative Lie algebra.
\begin{proposition}
Let $G$ be a multiplicative Lie algebra and $I$ be an ideal of $G$ contained in the center of $G$. Then $\Delta = \{(a^{-1}, a) \mid a\in I\}$ is an ideal of $G \oplus I$ but not an ideal of $G\rtimes I$ in general.
\end{proposition}
\begin{proof}
   It is clear that $\Delta$ is a normal subgroup of $G \oplus I$. So, let $(h,b) \in G \oplus I$ and $(a^{-1}, a)\in \Delta$. Then $(h,b)\star' (a^{-1}, a) = (h\star a^{-1}, (h\star a)(b\star a^{-1})(b\star a) ) = ((h\star a)^{-1},(h\star a)) \in \Delta $. Hence, $\Delta$ is an ideal of $G \oplus I$. However, for $(h,b) \in G \rtimes I$, we have $(h,b)\star'' (a^{-1}, a) = (h\star a^{-1}, (h\star a)(b\star a^{-1})) \notin \Delta$, in general. Hence, $\Delta$ need not be an ideal of $G\rtimes I$.
\end{proof}
The following theorem gives a necessary and sufficient condition for the structure of the ideals of $G \oplus I$ and $G\rtimes I$.
\begin{theorem}\label{P3}
Let $K$ and $J$ be two ideals of $G$ such that $J\subseteq I$. Then $K \oplus J~ (K \rtimes J)$ is an ideal of $G \oplus I~ (G\rtimes I)$ if and only if $I\star K \subseteq J$. 
\end{theorem} 
\begin{proof}
Suppose $K \oplus J$ is an ideal of $G \oplus I$. Then for $h\in K$and $a\in I$, $(1, a) \star'(h, 1)\in K \oplus J$. Thus $(1, a\star h) \in K \oplus J$, that is, $a\star h\in J$. Hence, $I\star K \subseteq J$. Conversely, suppose that $I\star K \subseteq J$. Let $(g,a)\in G\oplus I $ and $(h,b)\in K \oplus J$.  Since $I\star K \subseteq J$, we have $(g,a) \star' (h,b) = (g \star h, (g \star b)(a \star h)(a\star b))  \in K \oplus J$. Hence, $K \oplus J$ is an ideal of $G \oplus I$. Similarly, one can prove for $G\rtimes I$.
\end{proof}

\begin{corollary}\label{P4}
 Let $K$ be an ideal of $G$.  Then $K\oplus I~ (K \rtimes I)$ and $K\oplus (K \cap I)~ (K \rtimes (K \cap I))$ are ideals of $G\oplus I ~ (G\rtimes I)$. 
\end{corollary}
In the following proposition, we identify specific ideals within the excision multiplicative Lie algebra. 

\begin{proposition}\label{P4}
Let $G\oplus I $ be the excision multiplicative Lie algebra. Then
\begin{enumerate}
\item $[G \oplus I,G \oplus I]=[G,G]\oplus\{1\}$.
\item $Z(G \oplus I) = Z(G) \oplus I = Z(G) \oplus (Z(G)\cap I)$.
\item  $LZ(G \oplus I) = LZ(G) \oplus (LZ(G)\cap I)$.
\item $\mathcal{Z}(G \oplus I) = \mathcal{Z}(G) \oplus (\mathcal{Z}(G)\cap I) $.
\item $MZ(G \oplus I) = MZ(G) \oplus (MZ(G)\cap I)$.
\end{enumerate}

\end{proposition}

\begin{proof}
\begin{enumerate}[\itshape(1)]
\item $[G \oplus I,G \oplus I]= [G,G]\oplus [I,I] =[G,G]\oplus\{1\}.$ 
\item $Z(G \oplus I) = Z(G) \oplus Z(I) = Z(G) \oplus I.$ 
\item   Let $(g,a) \in LZ(G \oplus I)$. Then $(g,a)\star' (h,b) = (1,1)$ for all $(h,b) \in G\oplus I$. This implies $((g\star h),(g \star b)(a\star h)(a\star b))= (1,1)$, which gives $  (g\star h) =1 $ for all $h\in G$. Thus, we have $ g\in LZ(G)$. Also, $(g\star b)(a\star h)(a\star b)= 1$  for all $ h\in G$ and $b\in I$. In particular, $(a\star h)(a\star 1)= 1$  for all $ h\in G$, which implies $a \in LZ(G)$.  Hence, $ (g,a) \in  LZ(G) \oplus (LZ(G)\cap I) $.

Next, suppose $(g,a) \in LZ(G) \oplus (LZ(G)\cap I) $. Then $ (g\star h) = 1 = (a\star h) $ \ for all  $h\in G $. 
Thus, for $(h,b)\in G\oplus I$, we have
$(g,a)\star' (h,b) = ((g\star h),(g \star b)(a\star h)(a\star b))=(1,1)$. This shows that $ (g, a)\in LZ(G \oplus I)$. This completes the proof. 

\item By combining $\textit{(2)}$ and $\textit{(3)}$ to get the required result.

\item  Let $(g,a)\in MZ(G \oplus I).$ Then $(g,a)\star' (h,b)=[(g,a),(h,b)]$ for all $(h,b) \in G\oplus I.$ This implies that $(g\star h, (g\star b)(a\star h)(a\star b))= ([g,h], [a,b]).$ Thus, we have  $g\star h = [g,h]$ for all $h\in G$ and so $ g \in MZ(G).$ Also, $(g\star b)(a\star h)(a\star b)=[a,b]$ for all $b\in I$ and $h\in G.$  In particular, we have $(g\star 1) (a\star h)(a\star 1)= 1$ for all $h\in G.$ This gives 
$  a \star h = 1 = [a,h]$ for all $h \in G$ and so $a \in MZ(G)\cap I.$ Hence, $ (g,a) \in  MZ(G) \oplus (MZ(G)\cap I) .$

 For the reverse inclusion, suppose $(g,a) \in MZ(G) \oplus (MZ(G)\cap I) .$ Then, we have $g\star h = [g,h]$ and $ a \star h = [a,h] = 1$ for all $h\in G.$ Let $(h,b) \in G\oplus I.$ Then $(g,a)\star' (h,b) = ((g\star h),(g \star b)(a\star h)(a\star b)) = ([g,h],[g,b]) = ([g,h],1) = ([g,h],[a,b]) =[(g,a),(h,b)]$ since $I\subseteq Z(G).$ This proves the result.
\end{enumerate}
\end{proof}

\begin{remark}
    Proposition \ref{P4} also holds for idealization $G\rtimes I$.
\end{remark}

Now, we illustrate  Theorem \ref{P3}  with the help of an example. 
\begin{example}
Consider $V_4 = \langle a,b:a^2 = b^2 = 1, ab = ba \rangle$, with a multiplicative Lie algebra structure  $ a \star b = a $ \cite[Theorem 2.5]{MS}. Let $G=V_4 \times V_4$, and let $I=\{ (1,1),(1, a), (a,1), (a, a) \} $ be an ideal of $G$.  Suppose $K=\{ (1,1), (a,1)\}$ and $J=\{ (1,1), (1, a) \}$ are proper ideals of $G$.  It is straightforward to verify that $K\oplus J~~ (K \rtimes I) $ is an ideal of $G\oplus I ~~(G\rtimes I)$. Here, we note that $I\star K = \{(1,1)\}\subseteq J \subseteq I$.
\end{example}

The following example shows the necessity of the condition $I\star K \subseteq J\subseteq I$ in Theorem \ref{P3}.
\begin{example}
Consider the Dihedral group $D_4 = \langle x,y ~|~ x^4 = 1= y^2 , xy = yx^{-1}\rangle = G$ with a multiplicative Lie algebra structure $x \star y = x$ \cite[Theorem 2.5]{MS}, and take the ideal $I = D_4 \star D_4 = \langle x \rangle$ of $D_4$ . Let $K = D_4$ and $J = [D_4,D_4]= \langle x^2 \rangle$ be ideals of $D_4$. For $y\in K$ and $x\in I$, we have $ x \star y= x \notin J$, and so $I\star K \nsubseteq J$. We claim that $K\oplus J$ is not an ideal of $G \oplus I$. For that, let $(x,x) \in (G \oplus I)$ and $(y,1)\in ( K \oplus J)$. Then
 $(x,x) \star' (y,1)= (x \star y, (x \star 1)(x \star y)(x \star 1))=(x,x) \notin K \oplus J$ as $x \notin J$. 
\end{example}

The following example shows that not every ideal of $G\oplus I $ is of the form $K \oplus J$ as in Theorem \ref{P3}.
\begin{example}\label{E2}
Take $G = V_4 =  \langle a,b:a^2 = b^2 = 1, ab = ba \rangle = I$ with multiplicative Lie algebra structure  $ a \star b = a $ \cite[Theorem 2.5]{MS}. We know that $N = \{(1,1),(a, a), (b,b), (ab, ab)\}$ is a normal subgroup of $V_4\oplus V_4$. Let $(x,y)\in V_4\oplus V_4$ and $(z,z)\in N$. Then, it is easy to see that $(x,y)\star' (z,z)\in N$. Hence, $N$ is an ideal of $G\oplus I$, not of the form $K \oplus J$ as in Theorem \ref{P3}.

On the other hand, we see that $N$ is also a  normal subgroup of
$V_4\rtimes V_4$ but not an ideal of $V_4\rtimes V_4$ as $(a,b)\star''(a,a) = (a\star a, (a\star a)(b\star a)) = (1, a) \notin N$. Thus, $V_4\oplus V_4$ is not isomorphic to $V_4\rtimes V_4$.

\begin{remark}
 Note that $V_4\cong \Z_2 \times \Z_2 $ is a $2$-dimensional Lie algebra over the field $\Z_2$ with respect to the bracket operation $[a,b]=a$ as given in Example \ref{E2}. Since Example \ref{E2} implies that $V_4\oplus V_4$ and $V_4\rtimes V_4$ are not isomorphic.   So, we have two distinct Lie algebras of dimension $4$, namely $V_4\oplus V_4$ and $V_4\rtimes V_4$ over $\Z_2$.   
\end{remark}

\end{example}
The next example shows  that  every ideal of 
$G\rtimes I$ is also not of the form $K \rtimes J$ as in Theorem \ref{P3}.
\begin{example}
    Consider the cyclic group $G = \{1,x\} = I$. We know there is only a trivial multiplicative Lie algebra structure on $G$. Then $N = \{(1,1), (x,x)\}$ is an ideal of $G\rtimes I$ which is not of the form $K \rtimes J$ as in Theorem \ref{P3}.
\end{example}
\begin{remark}
In general, if $G$ is a non-trivial abelian group with trivial 
multiplicative Lie algebra structure, then the diagonal subgroup $\triangle = \{(g,g): g\in G\}$ forms an ideal of $G\rtimes G$ ($G\oplus G$) which is not of the form $K \rtimes J$  ($K\oplus J$)as in Theorem \ref{P3}.    
\end{remark}

\section{Iteration of the construction of $G\oplus I$} 
Let $G$ be a multiplicative Lie algebra, and $I$ is an ideal of $G$ such that  $I\subseteq Z(G)$. Consider the excision multiplicative Lie algebra $G\oplus I$. Since $\{1\}\oplus I$ is an ideal of $G\oplus I$ such that  $\{1\}\oplus I \subseteq Z(G\oplus I)$, we can define again  excision multiplicative Lie algebra $G\oplus_2 I = (G\oplus I)\oplus (\{1\}\oplus I)$. The operations in $G\oplus_2 I$ are given as follows:
\begin{align*}
   & ((g,a_1),(1,a_2))\cdot ((h,b_1),(1,b_2)) = ((gh,a_1b_1),(1,a_2b_2))\\
    &((g,a_1),(1,a_2))\tilde\star ((h,b_1),(1,b_2))\\
    &= \big((g\star h, (g\star b_1)(a_1\star h)(a_1\star b_1)), (1,(g\star b_2)(a_1\star b_2)(a_2\star h)(a_2\star b_1)(a_2\star b_2))\big)
\end{align*}
    
where $g,h\in G$ and $a_1, a_2,b_1, b_2\in I$.  

Inductively, we can define $G\oplus_n I$ for every $n.$

Also, it is easy to see that the set $G\oplus^2 I:= \{(g,(a_1,a_2)):  g\in G, a_1,a_2 \in I\}$  forms a multiplicative Lie algebra with respect to the following two binary operations:
    \begin{align*}
    (g,(a_1,a_2))\cdot (h,(b_1,b_2)) &:= (gh,(a_1b_1,a_2b_2)) \\
    (g, (a_1,a_2))\bar\star (h,(b_1,b_2)) &:= \big(g\star h, ((g\star b_1)(a_1\star h)(a_1\star b_1),(g\star b_2)(a_2\star h)(a_2\star b_2))\big).
    \end{align*}
Similarly, we can define $G\oplus^n I$ for every $n.$
\begin{theorem}\label{itera}
    The multiplicative Lie algebras $G\oplus_2 I$ and $G\oplus^2 I$ are isomorphic.
\end{theorem}
    \begin{proof} Define a map $\phi: G\oplus_2 I\to G\oplus^2 I  $ as  $$\phi((g,a_1),(1,a_2)) = (g,(a_1,a_1a_2))$$
Let $((g,a_1),(1,a_2)), ((h,b_1),(1,b_2))  \in G\oplus_2 I$. Then 
\begin{align*}
    \phi \big(((g,a_1),(1,a_2))((h,b_1),(1,b_2))\big) &= \phi \big(((gh,a_1b_1),(1,a_2b_2))\\
    &= (gh, (a_1b_1, a_1b_1a_2b_2))
    \\
    &= (g, (a_1, a_1a_2))(h, (b_1, b_1b_2))\\
    &=\phi((g,a_1),(1,a_2))\phi((h,b_1),(1,b_2))
\end{align*}
This shows that $\phi$ is a group homomorphism.
\begin{align*}
    &\phi \big(((g,a_1),(1,a_2))\tilde\star((h,b_1),(1,b_2))\big)\\ &= \phi\big((g\star h, (g\star b_1)(a_1\star h)(a_1\star b_1)), (1,(g\star b_2)(a_1\star b_2)(a_2\star h)(a_2\star b_1)(a_2\star b_2))\big)\\
    &= \big(g\star h, ((g\star b_1)(a_1\star h)(a_1\star b_1), (g\star b_1)(a_1\star h)(a_1\star b_1)(g\star b_2)(a_1\star b_2)(a_2\star h)(a_2\star b_1)(a_2\star b_2))\big)
    \\
    &= (g, (a_1, a_1a_2))\bar\star(h, (b_1, b_1b_2))\\
    &=\phi((g,a_1),(1,a_2))\bar\star\phi((h,b_1),(1,b_2))
\end{align*}
Thus, $\phi$ is a multiplicative Lie algebra homomorphism. It is easy to verify that the map $\psi: G\oplus ^2 I\to G\oplus_2 I $ defined as $(g,(a_1,a_2))\mapsto ((g,a_1),(1,a_2a^{-1}_1))$ is the inverse multiplicative Lie algebra homomorphism of $\phi.$ 
\end{proof}
\begin{remark}
    As Theorem \ref{itera}, one can also prove that $G\oplus_n I$ and $G\oplus^n I$ are isomorphic for every $n\in \mathbb{N}.$ For example, the map $\psi: G\oplus_3  I\to G\oplus ^3 I $ defined as $(g,(a_1,a_2, a_3))\mapsto (((g,a_1),(1,a_2a^{-1}_1)), ((1,1),(1,a_3a^{-1}_2))$ gives the multiplicative Lie algebra isomorphism with the inverse map  $\phi: G\oplus ^3  I\to G\oplus_3  I $ defined as $(((g,a_1),(1,a_2)), ((1,1),(1,a_3))\mapsto (g,(a_1,a_1a_2, a_2a_3))$
\end{remark}


\section{Excision as a Fiber Product}
Let $G_1, G_2$ and $H$ be multiplicative Lie algebras, and $\phi_1:G_1\to H$ and $\phi_2:G_2\to H$ be multiplicative Lie algebra homomorphisms. The fiber product of $G_1$ and $G_2$
over  $H$ is defined as:
$$G_1 \times_{H} G_2 = \{(g_1,g_2)\in G_1 \times G_2 \ | \ \phi_1(g_1)=\phi_2(g_2) \}.$$ Then $G_1 \times_{H} G_2$ is a subalgebra of $G_1 \times G_2$. In other words, the following diagram commutes:
\begin{center}
$\begin{CD}
{G_1 \times_{H} G_2} @>{p_2}>> G_2\\
@VV{p_1}V @VV{\phi_2}V\\
G_1 @>{\phi_1}>> H
\end{CD}$
\end{center}
where $p_1$ and $p_2$ are the first and second projections, respectively.

The following result establishes an isomorphism of excision multiplicative Lie algebra with certain fiber products. 
\begin{theorem}\label{T2}
Let $G\oplus I $ be the excision multiplicative Lie algebra. Then
$G\oplus I $ is isomorphic to the following fiber products:
\begin{enumerate}
    \item  $G \times_{\frac{G}{I}} G = \{(g,g')\in G \times G \ | \ g = g'(\text{mod}~ I) \} 
     = \{(g,ga)\in G \times G \ | \  a\in I\}.$ 

    \item $(G \times G)\times_{(G\times \frac{G}{I})} G.$

    \item $(G \times G)\times_{(\frac{G}{I}\times \frac{G}{I})} \frac{G}{I}.$
\end{enumerate}
 
\end{theorem}

\begin{proof}
\begin{enumerate}[\itshape(1)]
    \item 
Define a map $f_1:G \times_{\frac{G}{I}} G\to G\oplus I $  such that $(g,ga)\mapsto (g,a)$. It is easy to see that $f_1$ is a surjective group homomorphism and $\ker f_1=(1,1)$. Also, 
\begin{align*}
    f_1((g_1,g_1a_1)\star (g_2,g_2a_2)) &= f_1(g_1\star g_2,(g_1a_1)\star(g_2a_2)) \\
    &= f_1(g_1\star g_2, (g_1\star g_2)^{g_2}(g_1\star a_2)^{g_1}(a_1\star g_2)^{g_1g_2}(a_1\star a_2)) \\
    &= f_1(g_1\star g_2, (g_1\star g_2)(g_1\star a_2)(a_1\star g_2)(a_1\star a_2))  ~~   (\text{ By Lemma} ~\ref{L1}) \\
    &= (g_1\star g_2, (g_1\star a_2)(a_1\star g_2)(a_1\star a_2)) = (g_1,a_1)\star' (g_2,a_2)\\
    &= f_1((g_1,g_1a_1))\star' f_1((g_2,g_2a_2)).
\end{align*}
This implies that $f_1$ is a multiplicative Lie algebra isomorphism. In fact, $G\oplus I $ is isomorphic to a subalgebra of $G \times G$. 

\item We denote $(G \times G)\times_{(G\times \frac{G}{I})} G$ by $ \mathcal{G}_2 $. Then, we have the following commutative diagram:
\begin{center}
$\begin{CD}
\mathcal{G}_2 @>{p_2}>> G\\
@VV{p_1}V @VV{\phi}V\\
G \times G @>{\psi}>> {G\times \frac{G}{I}.}
\end{CD}$
\end{center}
That means 
\begin{align*}
    \mathcal{G}_2  &= \{((g,g'),g'')\in (G \times G)\times G \ | \ \psi(g, g')= \phi(g'') \} \\
     &= \{((g,g'),g'')\in (G \times G)\times G \ | \ (g, g'I)= (g'',g''I) \} = \{((g,gb),g)\in (G \times G)\times G \ | \ b\in I \}.
\end{align*}
Define a map $f_2:\mathcal{G}_2\to G\oplus I $  such that $((g,gb),g)\mapsto (g,b)$. It is easy to see that $f_2$ is a surjective group homomorphism and $\ker f_2=(1,1)$. Next, 
\begin{align*}
    & f_2\Big(((g_1,g_1b_1),g_1)\star ((g_2,g_2b_2),g_2)\Big) \\
    &= f_2\Big((g_1\star g_2,(g_1b_1)\star(g_2b_2)),g_1\star g_2\Big) \\
    &= f_2\Big((g_1\star g_2, (g_1\star g_2)^{g_2}(g_1\star b_2)^{g_1}(b_1\star g_2)^{g_1g_2}(b_1\star b_2)),g_1\star g_2\Big) \\
    &= f_2\Big((g_1\star g_2, (g_1\star g_2)(g_1\star b_2)(b_1\star g_2)(b_1\star b_2)),g_1\star g_2\Big)  ~~   (\text{ By Lemma} ~\ref{L1}) \\
    &= (g_1\star g_2, (g_1\star b_2)(b_1\star g_2)(b_1\star b_2))\\
    &= (g_1,b_1)\star' (g_2,b_2)\\
    &= f_2(((g_1,g_1b_1),g_1))\star' f_2(((g_2,g_2b_2),g_2)).
\end{align*}
This implies that $f_2$ is a multiplicative Lie algebra isomorphism. In particular, $G\oplus I $ is isomorphic to a subalgebra of $G \times G \times G$. 

\item Finally, we denote $(G \times G)\times_{(\frac{G}{I}\times \frac{G}{I})} \frac{G}{I}$ by $ \mathcal{G}_3 $. Then, we have the following commutative diagram:
\begin{center}
$\begin{CD}
\mathcal{G}_3 @>{p_2}>> \frac{G}{I}\\
@VV{p_1}V @VV{\mu}V\\
G \times G @>{\nu}>> {\frac{G}{I}\times \frac{G}{I}.}
\end{CD}$
\end{center}
That means 
\begin{align*}
    \mathcal{G}_3  &= \{((g,g'),g''I)\in (G \times G)\times \frac{G}{I} \ | \ \nu(g, g')= \mu(g''I) \} \\
     &= \{((g,g'),g''I)\in (G \times G)\times \frac{G}{I} \ | \ (gI, g'I)= (g''I,g''I) \}\\
     &= \{((g,gc),gI)\in (G \times G)\times \frac{G}{I} \ | \ c\in I \}.
\end{align*}
Now, the map $f_3:\mathcal{G}_3\to G\oplus I $ defined by $((g,gc),gI)\mapsto (g,c)$ can be easily seen as a multiplicative Lie algebra isomorphism. In particular, $G\oplus I $ is isomorphic to a subalgebra of $G \times G \times \frac{G}{I}$. 
\end{enumerate}
 
\end{proof}
\begin{remark}
\begin{enumerate}
    \item Since any subalgebra of a nilpotent (solvable) multiplicative Lie algebra is nilpotent (solvable), and a direct product of two nilpotent (solvable)  multiplicative Lie algebras is nilpotent (solvable) (see \cite{FP}), by Theorem \ref{T2}, $G\oplus I $ is nilpotent (solvable).

    \item Since any subalgebra of a Lie nilpotent (Lie solvable) multiplicative Lie algebra is Lie nilpotent (Lie solvable) and a direct product of two Lie nilpotent (Lie solvable)  multiplicative Lie algebras is Lie nilpotent (Lie solvable) (see \cite{MRS}), by Theorem \ref{T2}, $G\oplus I $ is Lie nilpotent (Lie solvable).
\end{enumerate} 
\end{remark}

\begin{proposition}
   Let $G$ be a nilpotent/solvable multiplicative Lie algebra. Then $G\rtimes I$ is also nilpotent/solvable. 
\end{proposition}
\begin{proof}
    Suppose $G$ is a nilpotent multiplicative Lie algebra of class $n$. Then $(((x_1\star x_2)\star x_3)\star \cdots \star x_{n+1} ) = 1$ for all $x_1, x_2, \cdots , x_{n+1}\in G$ (see \cite{FP}). Now, it follows from Theorem \ref{T1} that $(x_1,a_1)\star''(x_2,a_2)\star''\cdots \star'' (x_{n+1},a_{n+1}) = 1$ for all $(x_1,a_1), (x_2,a_2), \cdots (x_{n+1},a_{n+1})\in G\rtimes I$. Hence, $G\rtimes I$ is also nilpotent of class $n$.
\end{proof}
\begin{remark}
\begin{enumerate}
\item Let $L$ be a nilpotent (solvable) Lie algebra and $M$ be an ideal. Then $L\rtimes M$ and $L\oplus M$ are also nilpotent/solvable.
    \item If $G$ is a Lie nilpotent/Lie solvable multiplicative Lie algebra (see \cite{MRS}), one can easily show that $G\rtimes I$ is also Lie nilpotent/Lie solvable.
\end{enumerate}
    
\end{remark}
\noindent{\bf Acknowledgment:} The first named and third named author are thankful to the National Board for Higher Mathematics (NBHM) for providing the  project ``Linear Representation of Multiplicative Lie Algebra" (02011/19/2023/NBHM (R.P)/R~\& D-II/5954). 

\noindent{\bf  Data Availability declaration:} Not applicable.

\noindent{\bf Competing Interests:} The authors have no competing interests to declare that are relevant to the content of this article.

\noindent{\bf Funding Declaration:} The first named author is supported by the National Board for Higher Mathematics (NBHM), project ``Linear Representation of Multiplicative Lie Algebra" (02011/19/2023/NBHM (R.P)/R~\& D-II/5954). The second and third named  authors have no relevant financial or non-financial interests to disclose.


\end{document}